\pdfoutput=1

\documentclass[reqno]{amsart}

\usepackage[latin1]{inputenc}

\usepackage{microtype}
\usepackage{amssymb}
\usepackage{amsmath}
\usepackage{amscd}
\usepackage{amsthm}
\usepackage{amsfonts}
\usepackage{enumerate}
\usepackage{graphicx}
\usepackage{url}
\usepackage[breaklinks=true]{hyperref}
\usepackage{amssymb}
\usepackage[dvips]{color}
\usepackage{epsfig}
\usepackage{mathrsfs}

\newtheorem{theorem}{Theorem}[section]
\newtheorem{lemma}[theorem]{Lemma}

\newtheorem{corollary}[theorem]{Corollary}

\theoremstyle{definition}

\newtheorem{definition}[theorem]{Definition}

\theoremstyle{remark}

\newtheorem{remark}[theorem]{Remark}
\numberwithin{equation}{section}

\date{}
\author{Jon Schneider}
\title{Polynomial sequences of binomial-type arising in graph theory}
\begin{document}
\maketitle

\begin{abstract}
In this paper, we show that the solution to a large class of ``tiling'' problems is given by a polynomial sequence of binomial type. More specifically, we show that the number of ways to place a fixed set of polyominos on an $n\times n$ toroidal chessboard such that no two polyominos overlap is eventually a polynomial in $n$, and that certain sets of these polynomials satisfy binomial-type recurrences. We exhibit generalizations of this theorem to higher dimensions and other lattices. Finally, we apply the techniques developed in this paper to resolve an open question about the structure of coefficients of chromatic polynomials of certain grid graphs (namely that they also satisfy a binomial-type recurrence).
\end{abstract}

\section{Introduction}

A sequence $p_0 = 1, p_1, p_2, \dots$ of polynomials is a polynomial sequence of binomial type if it satisfies the identity

\begin{equation}
\left(\sum_{i=0}^{\infty}p_i(1)x^i\right)^n = \sum_{i=0}^{\infty}p_i(n)x^i
\end{equation}

Binomial-type sequences were introduced by Rota, Kahaner, and Odlyzko in 1975 \cite{RKO} and play an important role in the theory of umbral calculus. Outside of the context of umbral calculus, polynomial sequences of binomial type possess the useful property that they are completely determined by the sequence of their values when evaluated at a single point. Several important polynomial sequences, such as the Abel polynomials and the Touchard polynomials, are of binomial type.

In this paper we demonstrate that polynomial sequences of binomial type arise from a large class of problems occuring in graph theory. In particular, these sequences occur in problems where we wish to enumerate the number of ways to place some objects on a ``toroidal'' periodic structure such that no two overlap. One simple example of this phenomenon is the following. Take an $n \times n$ chessboard and identify opposite edges to make it toroidal. If we let $p_k(n^2)$ equal the number of ways to place $k$ dominoes on this grid (aligned with the grid's edges) such that no two dominoes overlap, then it turns out that, for sufficiently large $n$ the values $p_k(n^2)$ are given by a polynomial in $n^2$. Moreover, this sequence of polynomials (viewed as a sequence in $k$) is a polynomial sequence of binomial type.

Our methods allow us to easily generalize these results. The main result of our paper is a generalization of the above phenomenon to arbitrary sets of polyominos on toroidal grids of any dimension. We also demonstrate some interesting further generalizations of this result; for example, we show that the same binomial-type relation holds when we can assign arbitrary integer weights to polyominos and then count placements that have a total weight of $k$. In addition, we show that there is a very natural continuous analogue of these results concerning placing arbitrary bounded measurable `shapes' in a continous $d$-dimensional torus. 

Finally, we apply these results to provide a proof of an open problem due to Stanley \cite{St86} concerning coefficients of the chromatic polynomial $\chi_n(x)$ of the two-dimensional toroidal grid graph. Much previous research into the chromatic polynomial $\chi_n(x)$ of toroidal grid graphs focused primarily on the asymptotics of this polynomial, particularly the limit $\lim_{n \rightarrow \infty} (\chi_n(x))^{1/n^d}$ (see, for instance \cite{Ba71, Ba82, Bi75, CS02, CS04, KE79, Lieb, Na71}). For example, it is known that for $d = 2$ and $x = 3$, this limit is equal to $(4/3)^{3/2}$.

The open problem due to Stanley asks to show that the coefficient of $x^{n^2-k}$ in $\chi_n(x)$ is (up to sign) for sufficiently large $n$ a polynomial in $n^2$, and that the polynomials for different $k$ form a polynomial sequence of binomial type. By using Whitney's broken-circuit theorem to reduce this to an overlap problem of the style above, we resolve this open problem (and in fact, provide a generalization that holds for chromatic polynomials of toroidal grid graphs of any positive dimension). 

Our paper is structured as follows. In Section \ref{definitions}, we define some terminology that we use throughout this paper. In Section \ref{polynomial}, we prove that the problem of enumerating the number of non-overlapping placements does in fact give rise to a polynomial for sufficiently large $n$. We additionally show how to write these polynomials in a nice form reminiscent of certain generating functions. In Section \ref{ischemes}, we introduce the notion of an intersection schema and use it to prove our main theorem. In Section \ref{generalizations}, we discuss generalizations of our main result to the cases of assigning integer weights of polyominoes, continous tori, non-toroidal grids, and other types of lattices. Finally, in Section \ref{chrom}, we apply our main result along with Whitney's broken-circuit theorem to solve the open problem mentioned above. 

\section{Background and Definitions}\label{definitions}

We begin with some graph-theoretic notation. We let $C_n$ denote the cycle graph on $n$ vertices, $P_n$ denote the path graph on $n$ vertices, and $C_{\infty}$ denote the doubly infinite path graph. 

\begin{definition}
Given two graphs $G_1$ and $G_2$, we define the \textit{product graph} $G_1 \times G_2$ as follows. The vertices of $G_1 \times G_2$ are given by ordered pairs $(v_1, v_2)$ where $v_1 \in G_1$ and $v_2 \in G_2$. The two vertices $(v_1, v_2)$ and $(v'_1, v'_2)$ are adjacent if either $v_1 = v'_1$ and $v_2$ is adjacent to $v'_2$ in $G_2$ or $v_2 = v'_2$ and $v_1$ is adjacent to $v'_1$ in $G_1$. We write $G^r$ for the expression $G \times G \times \dots \times G$ (with $r$ copies of $G$).
\end{definition}

\begin{definition}
The \textit{$d$-dimensional toroidal grid graph of size $n$}, $T^{d}_n$, is the graph $(C_n)^d$. Similarly, the \textit{$d$-dimensional infinite toroidal grid graph} $T^{d}_{\infty}$ is the graph $(C_{\infty})^d$, and the \textit{$d$-dimensional grid graph of size $n$}, $L^{d}_n$, is the graph $(P_n)^d$.
\end{definition}

\begin{definition}
A \textit{$d$-dimensional figure} is a finite subset of vertices of $T^{d}_{\infty}$, up to translation. That is, two figures are considered equivalent if we can get from one to the other by adding a fixed integer vector to the coordinates of each of its vertices. We say a figure is of size $s$ if it contains $s$ vertices. We also say a figure is of girth $g$ if the maximum coordinate difference between two vertices in the figure is equal to $g$.
\end{definition}

We have defined figures above as subsets of $T^{d}_\infty$. However, it is clear that, for any specific figure, if $n$ is large enough (in particular, larger than the girth of the figure), then we can also view the figure as a subset of $T^{d}_n$ (up to translation). We will often abuse notation in this way by talking about ``placing'' figures on $T^{d}_n$. In such cases, we will always assume that we are taking $n$ large enough so that this makes sense. 

We next define what it means for a sequence of polynomials to be of binomial type.

\begin{definition}
The sequence of polynomials $\lbrace p_i(n)\rbrace_{i\geq 0}$ is of \textit{binomial-type} if it satisfies the following three properties: i. $p_0(n) = 1$, ii. $p_i(0) = 0$ for $i>0$, and iii. the identity given by equation (\ref{eqbintype}) below holds for all nonnegative $n$. 

\begin{equation}\label{eqbintype}
\left(\sum_{i=0}^{\infty}p_i(1)x^i\right)^n = \sum_{i=0}^{\infty}p_i(n)x^i
\end{equation}

An equivalent reformulation of our third condition is that the identity given by equation (\ref{eqbintype2}) below holds for all nonnegative $n$.

\begin{equation}\label{eqbintype2}
p_n(x+y) = \sum_{i=0}^{n}p_i(x)p_{n-i}(y)
\end{equation}

\end{definition}

Our definition of binomial-type differs slightly from the definition most often found in the literature, where equation (\ref{eqbintype2}) contains an additional factor of $\binom{n}{i}$. The two definitions are easily interchangeable, however; if $\lbrace p_i(n)\rbrace$ is a sequence of binomial-type under our definition, then $\lbrace i!p_i(n) \rbrace$ is a sequence of binomial-type under the traditional definition.

\section{Polynomiality}\label{polynomial}

In this section, we demonstrate that several sequences related to the number of ways to place a fixed set of figures on a lattice are eventually described by polynomials. More specifically, we have the following main result.

\begin{theorem}\label{polymain}
Let $S$ be a finite multiset of $d$-dimensional figures. Let $f_{S}(n)$ be the number of ways to place all of the figures in $S$ on $T^{d}_n$ such that none overlap (for a finite set of small values of $n$, there may be figures that are impossible to place on $T^{d}_n$; in this case, let $f_{S}(n) = 0$). Then there exists a positive integer $n_0$ and an integer polynomial $p(x)$ such that $f_{S}(n) = p(n^d)$ for all $n \geq n_0$. 
\end{theorem}

Since there is some subtlety in dealing with multisets containing repeated indistinguishable figures, in the first half of this section (Subsection \ref{norepeats}) we prove this result only for sets of distinct figures. In the second half (Subsection \ref{yesrepeats}), we generalize to the case where repeats of figures are allowed.

\subsection{Without repeats}\label{norepeats}

In this subsection, we prove Theorem \ref{polymain} for the case where $S$ contains no repeated figures. In particular, we prove the following simpler result.

\begin{theorem}\label{polymainredux}
Let $S$ be a finite set of \textbf{distinct} $d$-dimensional figures. Let $f_{S}(n)$ be the number of ways to place all of the figures in $S$ on $T^{d}_n$ such that none overlap (for a finite set of small values of $n$, there may be figures that are impossible to place on $T^{d}_n$; in this case, let $f_{S}(n) = 0$). Then there exists a positive integer $n_0$ and an integer polynomial $p(x)$ such that $f_{S}(n) = p(n^d)$ for all $n \geq n_0$. 
\end{theorem}

Throughout this subsection and the next, we will repeatedly make use of the following notion of an overlap graph.

\begin{definition}
An \textit{overlap graph} is a graph $G$ whose vertices are labelled by $d$-dimensional figures (for some $d$). A placement of these figures on $T^{d}_n$ (or $T^{d}_{\infty}$) is \textit{consistent} with $G$ if, whenever figures $f_1$ and $f_2$ are adjacent in $G$, they overlap in $T^{d}_n$ (or $T^{d}_{\infty}$). 
\end{definition}

Note that if a placement of figures is consistent with an overlap graph, so are all translations of this placement of figures. This inspires the following definition. 

\begin{definition}
A \textit{configuration} of $d$-dimensional figures is a placement of $d$-dimensional figures on $T^{d}_n$ (or $T^{d}_{\infty}$) where two configurations are equivalent if they are translations of each other in $T^{d}_{n}$ (or $T^{d}_{\infty}$. A configuration $c$ is \textit{consistent} with a graph $G$ if any of its placements are consistent with $G$; in this case, we write $c \unlhd G$.
\end{definition}

If $f_1$ and $f_2$ are not adjacent in $G$, they may or may not overlap in $T_n^{d}$ (or $T^{d}_{\infty}$); only one direction of the above implication holds. In addition, for now we will assume that the vertices of our overlap graphs are labelled with distinct $d$-dimensional figures; we will lift this constraint in the following subsection.

We next prove three useful lemmas about overlap graphs.

\begin{lemma}\label{polylem1}
If an overlap graph $G$ is connected, then, there are only finitely many configurations of these figures on $T^{d}_{\infty}$ consistent with $G$. We call this number $v(G)$.
\end{lemma}
\begin{proof}
Since $G$ is connected, there exists a path of edges of $G$ between any two vertices of $G$. This implies that, for any two figures $f_1$ and $f_2$ in a consistent placement of these figures on $T^{d}_{\infty}$, we can construct a sequence of figures starting at $f_1$ and ending at $f_2$ such that each figure intersects the next figure in the sequence.

Now, since each of the figures has finite size, and since there are a finite number of figures, this implies that the maximum distance (along edges of the graph) between any two points belonging to figures in our placement is bounded. Since there are only finitely many ways to place a finite number of figures in a bounded region of $T^{d}_{\infty}$, this establishes that $v(G)$ is finite.
\end{proof}

\begin{lemma}\label{polylem2}
Let $G$ be a connected overlap graph. Then the number of placements of figures on $T^{d}_{n}$ consistent with $G$ is equal to $v(G)n^d$ for sufficiently large $n$.
\end{lemma}
\begin{proof}
By Lemma \ref{polylem1}, we know that there are $v(G)$ distinct consistent configurations of these figures on $T^{d}_{\infty}$. For sufficiently large $n$, it will be possible to embed each of these $v(G)$ configurations in $T^{d}_{n}$. Finally, for each choice of consistent configuration, there are $n^d$ possible translations in $T^{d}_{n}$. Therefore, there are a total of $v(G)n^d$ consistent placements on $T^{d}_n$ for sufficiently large $n$.
\end{proof}

\begin{lemma}\label{polylem3}
Let $G$ be an overlap graph with connected components $G_1, G_2, \dots, G_r$. Then, the number of placements of figures on $T^{d}_n$ consistent with $G$ is equal to $v(G_1)v(G_2)\dots v(G_r)n^{rd}$.
\end{lemma}
\begin{proof}
We first note that we can treat all the connected components ``independently''. More specifically, for each $i$ let $P_i$ be a placement of figures on $T^{d}_n$ consistent with $G_i$. Then the placement of figures given by the union of the $P_i$ is a placement consistent with $G$ (and moreover, all consistent placements with $G$ can be written uniquely in such a way). This follows directly from the fact that, since there are no edges between figures belonging to different connected components of $G$, there are also no overlap constraints they must satisfy (in addition to the fact that all of our figures are distinct, so we can identify which connected component of $G$ they must belong to). 

Therefore, since by Lemma \ref{polylem2}, there are $v(G_i)n^{d}$ placements on $T^{d}_n$ consistent with $G_i$  (for sufficiently large $n$), overall there will be

\begin{equation}
\prod_{i=1}^{r}v(G_i)n^{d} = \left(\prod_{i=1}^{r}v(G_i)\right)n^{rd}
\end{equation}

\noindent
placements of figures on $T^{d}_n$ consistent with $G$ (for sufficiently large $n$), as desired.
\end{proof}

With these lemmas, the proof of Theorem \ref{polymainredux} reduces to a straightforward application of the principle of inclusion-exclusion.

\begin{proof}[\textbf{Proof of Theorem~\ref{polymainredux}}]
Label the figures in $S$ as $f_1, f_2, \dots, f_m$. We wish to count the number of placements of these figures on $T^{d}_n$ such that no two figures overlap. Thus, for each pair $1 \leq i < j \leq m$, let $E(i,j)$ be the set of placements of these figures where $f_i$ and $f_j$ intersect, and let $U$ be the set of all placements of these figures. Finally, for convenience of notation, let $\mathcal{P}$ be the set of all pairs $(i,j)$ where $1\leq i < j \leq m$; if $p = (i,j)$, we will also let $E(p)$ stand for $E(i,j)$.

Then, by the principle of inclusion-exclusion, the number of placements where no two figures overlap is equal to

\begin{equation}\label{pieeq}
|U| - \sum_{p_1\in\mathcal{P}}|E(p_1)| + \sum_{p_1, p_2 \in\mathcal{P}}|E(p_1)\cap E(p_2)| - \sum_{p_1,p_2,p_3 \in\mathcal{P}}|E(p_1)\cap E(p_2) \cap E(p_3)| + \cdots
\end{equation}

Hence, to show that this is eventually a polynomial in $n^d$, it suffices to show that each of the individual terms is eventually a polynomial in $n^d$. Now, $|E(p_1)\cap E(p_2) \cap \dots \cap E(p_k)| = |E(i_1, j_1) \cap E(i_2, j_2) \cap \dots \cap E(i_k, j_k)|$ is equal to the number of placements where figure $f_{i_r}$ intersects figure $f_{j_r}$ for each $r$ between $1$ and $k$. But this is simply equal to the number of placements consistent with the overlap graph $G$ which contains an edge between $f_{i_r}$ and $f_{j_r}$ for each $r$ between $1$ and $k$. By Lemma \ref{polylem3}, this is eventually a polynomial in $n^d$.

We also have the term $|U|$, consisting of all possible placements. But this is also just equal to the number of placements consistent with the overlap graph $G$ containing no edges, so once again by Lemma \ref{polylem3}, this is also a polynomial in $n^d$ (in fact, we have that $|U| = (n^{d})^m$). This concludes the proof.
\end{proof}

For a connected overlap graph $G$, let $a(G) = (-1)^{|E|}v(G)$. By substituting values into Equation \ref{pieeq} from Lemma \ref{polylem3}, we have the following corollary.

\begin{corollary}
Let $N = n^d$. Then, for sufficiently large $n$, the function $f_{S}(n)$ defined in Theorem \ref{polymainredux} is equal to

\begin{equation}\label{isform}
\sum_{r = 1}^{m}\sum \dfrac{1}{r!}a(g_1)a(g_2)\cdots a(g_r)N^r
\end{equation}
\noindent
where the inner sum is over all \textbf{ordered} $r$-tuples of connected overlap graphs that union to an overlap graph for the set $S$ (equivalently, the union of the sets of figures corresponding to the vertices of the $g_i$ is equal to the set $S$).
\end{corollary}

Note that since $S$ contains only distinct figures, we could easily remove the factor of $1/r!$ in equation (\ref{isform}) and instead sum over all unordered $r$-tuples. However, for reasons to be explained in Section \ref{ischemes}, it is more convenient to write our polynomial in this form.

\subsection{With repeats}\label{yesrepeats}

In the previous section, we proved Theorem \ref{polymain} for the specific case where $S$ contained no repeated indistinguishable figures. In the case that we have several of the same figure, certain details in the above proof (in particular, Lemma \ref{polylem3}) fail to hold. For example, if our set $S$ contains two identical figures, the number of total possible placements is no longer $n^{2d}$ (nor is it $n^{2d}/2!$, since this is not even always an integer). Instead, it is equal to $(n^{2d} + n^{d})/2$; the extra $n^d$ term arises from the fact that our two indistinguishable figures can occupy exactly the same location. 

However, if we have repeated \textit{distinguishable} figures, all of the logic in the previous section continues to hold. For instance, if we have two identical figures, but of which one is colored red and the other blue, then there are again $n^{2d}$ possible placements of these two figures. This observation gives rise to a simple proof of Theorem \ref{polymain}.

\begin{proof}[\textbf{Proof of Theorem~\ref{polymain}}]
Assume that $S$ contains $c_i$ copies of figure $f_i$, for each $1 \leq i \leq m$.

For each group of indistinguishable repeated figures in $S$, ``color'' them to make them distinguishable. Then the proof of Theorem \ref{polymainredux} implies that the number of ways to place these figures on $T^{d}_{n}$ such that no two figures overlap is eventually some polynomial $p(n^d)$ for large enough $n$.

But now, if we ignore the different colors, each configuration where no two figures overlap is counted exactly $c_1!c_2!\dots c_m!$ times (keep in mind that, since figures cannot overlap in these configurations, they cannot occupy exactly the same location). Therefore in the case of indistinguishable repeated figures, the number of placements of these figures such that no two figures overlap is eventually $p(n^d)/(c_1!c_2!\dots c_m!)$, which is also a polynomial in $n^d$.
\end{proof}

For reasons that will be explained in the next section, we would also like to write this polynomial in the same form as equation (\ref{isform}). In order to construct the correct function $a(g)$, we must introduce some more notation.

\begin{definition}\label{Osets}
For any configuration $c$, we can partition the figures of $c$ into $k$ maximal sets $O_i$ such that all the figures in $O_i$ are identical and overlap completely. Then the \textit{weight} $w_c$ of a configuration $c$ is defined to equal $\prod_{i=1}^{k}(o_i!)^{-1}$, where $o_i = |O_i|$. 
\end{definition}

\begin{definition}\label{alphadef}
In a connected overlap graph $G$, assume that there are a total of $c_i$ (vertices labeled with) figures of type $f_i$ for each $1 \leq i \leq m$. Then we define $\alpha(G) = |\mathrm{Aut}(G)|^{-1}\prod_{i=1}^{m}c_i!$, where $\mathrm{Aut}(G)$ is the group of automorphisms of the graph $G$ that preserve labelling (that is, they send vertices labelled with figures of type $f_i$ to vertices labelled with figures of type $f_i$). 
\end{definition}

Note that if we let $H$ be the group of all permutations of the vertices of $G$ which send figures of type $f_i$ to figures of type $f_i$, then $\alpha(G)$ can be equivalently defined as $|H|/|\mathrm{Aut}(G)|$. Similarly, this is just the number of non-isomorphic ways to color each set of $c_i$ figures of type $f_i$ with $c_i$ distinguishable colors. We will make use of this fact in the proof of the following theorem.

\begin{theorem}\label{mainform}

For a connected overlap graph $g$, define

\begin{equation} \label{genA}
a(g) = (-1)^{|E|}\alpha(g)\sum_{c \unlhd g}w_c
\end{equation}
\noindent
where the sum is over all configurations $c$ consistent with $g$. Then, as before, we have that

\begin{equation}\label{isform}
f_{S}(n) = \sum_{r = 1}^{m}\sum \dfrac{1}{r!}a(g_1)a(g_2)\cdots a(g_r)N^r
\end{equation}
\noindent
where the inner sum is over all \textbf{ordered} $r$-tuples of connected overlap graphs that union to an overlap graph for the set $S$.

\end{theorem}

\begin{proof}

Let $p$ be any placement of the figures of $S$ onto $T^{d}_n$. We will compute the number of times this placement is counted in the above sum, and show that this sum reduces to the similar sum that occurs in the case of completely distinguishable figures (as in Theorem \ref{polymainredux} above).  

To do this, for our placement $p$ of the figures in $S$, as in Definition \ref{Osets}, partition $S$ into $k$ maximal sets $O_i$ such that all the figures in $O_i$ are identical and overlap completely (and let $o_i = |O_i|$). In addition, let there be $c_i$ figures of type $f_i$ for each $i$ between $1$ and $m$.

Next, note that by the definition of $\alpha(G)$, for any connected overlap graph $g$, we can write $a(g) = \sum a'(h)$, where the sum is over all graphs $h$ obtained by coloring all the figures of type $f_i$ distinguishably (note that there are $\alpha(g)$ such graphs). Our new function $\alpha'(h)$ is now given just by $\alpha'(h) = (-1)^{|E|}\sum_{c \unlhd h} w_c$ (where for a configuration $c$ to be consistent with $h$, it simply has to be consistent with the original graph $g$). We can therefore write $f_{S}(n)$ as

$$f_{S}(n) = \sum_{r = 1}^{m}\sum \dfrac{1}{r!}a'(h_1)a'(h_2)\cdots a'(h_r)N^r$$
\noindent
where this new sum is over ordered $r$-tuples of these additionally colored overlap graphs $h_i$. Now, let us consider the terms of this sum that contribute to the total count for our placement $p$. Specifically, we are looking at terms where some subset $p_i$ of $p$ is counted in the placements belonging to the term $a'(h_i)N$, and such that the union of all the $p_i$ is $p$. Let us write one such term as

$$t_p = \dfrac{1}{r!}w_{p_1}w_{p_2}\cdots w_{p_r}$$
\noindent
where $w_{p_{i}}$ is the weight of placement $p_i$ (which is the same as the weight of the configuration $c_i$ to which $p_i$ belongs). Now, note that the overlap sets $O_i$ for our overall placement $p$ are distributed among these $r$ placements. Therefore, for each $1 \leq j \leq r$, let $O_{ij}$ be the subset of $O_{i}$ that belongs to subplacement $p_j$, and let $o_{ij} = |O_{ij}|$. Note then that $w_{p_i} = (o_{1i}!o_{2i}!\dots o_{ki}!)^{-1}$. We thus have that

$$t_p = \prod_{i=1}^{k}\prod_{j=1}^{r} (o_{ij}!)^{-1}.$$

Next, note that $o_{i}!\prod_{j=1}^{r} (o_{ij}!)^{-1}$ is the number of ways to split $o_i$ distinguishable colors and for each $j$ assign $o_{ij}$ of these colors to the placement $p_j$. Therefore, whereas $t_p$ was counting the (weighted) number of terms where we could distinguish between identical figures in each connected overlap subgraph, if we multiply all of these terms by $\prod_{i=1}^{k}o_i!$, we will now be counting over terms where we can distinguish among sets of identical figures that overlap completely.

Next, note that since each set $O_i$ contains figures of the same type, the expression 

$$\dfrac{\prod_{i=1}^{m}c_i!}{\prod_{i=1}^{k}o_i!}$$

\noindent
counts the number of ways to, for each $i$ between $1$ and $m$, split $c_i$ distinct colors and assign them to all the sets $O_i$ containing figures of type $f_i$. Therefore, if we again multiply all of these terms by $\prod_{i=1}^{m}c_i!/\prod_{i=1}^{k}o_i!$, we are now counting over terms where we can distinguish between any two figures; in particular, our sum is now exactly the same as it is in the above proof of Theorem \ref{polymain}. 

Altogether, we have multiplied the terms we are considering by a net factor of 

$$\left(\dfrac{\prod_{i=1}^{m}c_i!}{\prod_{i=1}^{k}o_i!}\right)\prod_{i=1}^{k}o_i! = \prod_{i=1}^{m}c_i!$$

But note that this does not depend on the specific placement $p$ we have chosen at all! Hence, we have shown that 

$$f_{S}(n)\prod_{i=1}^{m}c_i! = f_{\bar{S}}(n)$$

\noindent
where $\bar{S}$ is constructed from $S$ by coloring all the figures so that they are distinguishable. As shown above (again in the proof of Theorem \ref{polymain}), $f_{\bar{S}}(n)/\prod_{i=1}^{m}c_i!$ is exactly $f_{S}(n)$, and therefore our proof is complete. 

\end{proof}

\section{Intersection Schemas}\label{ischemes}

In the previous section, we showed that we can write the function $f_{S}(n)$ in the form given in equation (\ref{isform}). In particular, we have that

$$f_{S}(n) = \sum_{r = 1}^{m}\sum \dfrac{1}{r!}a(g_1)a(g_2)\cdots a(g_r)N^r$$

\noindent
where the function $a$ is given as in equation (\ref{genA}).

In this section, we will prove that a large class of functions written in this form give rise to polynomial sequences of binomial type. To do this, we will define an object called an \textit{intersection schema}, which will generalize many of the properties of overlap graphs we encountered in the previous section. 

\begin{definition}
A \textit{weighted set} is a set $S$ (either finite or infinite) along with a weight function $w: S \rightarrow \mathbb{Z}^+$  (from $S$ to the positive integers) such that for any $W$, there are only finitely many elements $x$ of $S$ such that $w(x) \leq W$. 
\end{definition}

\begin{definition}
Given a weighted set $S$, the set of $S$\textit{-labeled graphs} is the set of graphs where each vertex is labelled by an element in $S$. We shall denote this set as $LG(S)$. We define the weight $w(g)$ of an element $g$ of $LG(S)$ to simply be the sum of the weights of its labels. We will also denote the subset of $LG(S)$ consisting of connected $S$-labeled graphs as $LCG(S)$. Note that we consider two graphs in $LG(S)$ to be equivalent if they are equivalent under a graph isomorphism that sends labeled vertices to similarly labeled vertices.
\end{definition}

\begin{definition}\label{ISdef}
An \textit{intersection schema} is a weighted set $S$ along with a function $a: LCG(S) \rightarrow \mathbb{C}$. Given an intersection schema, we define the polynomial $q_i(N)$ as:

$$q_i(N) = \sum_{r = 1}^{\infty}\sum \dfrac{1}{r!}a(g_1)a(g_2)\cdots a(g_r)N^r$$

\noindent
where the inner sum is over all \textbf{ordered} $r$-tuples $(g_1,g_2,\dots,g_r)$ of elements of $LCG(S)$ such that $\sum_{j=1}^{r}w(g_j) = i$. By default, we set $q_0(N) = 1$. 
\end{definition}

Finally, we will need the following binomial identity:

\begin{lemma}\label{binlemma}
We have that

$$n^k = \sum_{m_1+\dots+m_r = k}\dfrac{k!}{m_1!m_2!\cdots m_r!}\binom{n}{r}$$

\noindent
where the sum is over all compositions of $k$.
\end{lemma}
\begin{proof}
The left hand side is the number of ways to color a set of $k$ items with $n$ colors. The right hand side counts this same number; here the $m_i$ correspond to sizes of sets of items colored the same color, the first multinomial coefficient corresponds to the number of ways to distribute the $k$ items into these groups of size $m_i$, and the binomial coefficient corresponds to the number of ways to choose $r$ colors for these $r$ groups out of the total $n$ colors.
\end{proof}

We can now state and prove our main theorem. 

\begin{theorem}\label{mainthm}
Let $\mathcal{I}$ be an intersection schema. We then have:

$$\left(\sum_{i=0}^{\infty}q_i(1)x^i\right)^N = \left(\sum_{i=0}^{\infty}q_i(N)x^i\right)$$
\end{theorem}
\begin{proof}
Let

$$F(x) = \sum_{i=0}^{\infty}q_i(1)x^i$$
\noindent
and let

$$F_N(x) = \sum_{i=0}^{\infty}q_i(N)x^i$$.

By the formula for $q_i(n)$ given in Definition \ref{ISdef}, we can rewrite $F(x)$ as:

$$F(x) = 1+\sum_{r=1}^{\infty}\sum \dfrac{1}{r!}a(g_1)a(g_2)\dots a(g_r)x^{w(g_1)+w(g_2)+\dots+w(g_r)}$$
\noindent
where the inner sum is over all ordered $r$-tuples $(g_1, g_2, \dots, g_r)$ of elements of $LCG(S)$.

We will now show by comparing terms that $F(x)^N = F_N(x)$. For sake of convenience, we will assume that the $a(g)$ are arbitrary non-commuting variables; i.e. that we do not necessarily have $a(g_1)a(g_2) = a(g_2)a(g_1)$ (of course, since $a(g_i) \in \mathbb{C}$, they do commute, but we will remove this restriction for now).

A general term in $F_n(x)$ looks like:

$$\dfrac{n^r}{r!}a(g_1)a(g_2)\dots a(g_r)x^{w(g_1)+w(g_2)+\dots+w(g_r)}$$

We will show that the coefficient of $a(g_1)a(g_2)\dots a(g_r)x^{w(g_1)+w(g_2)+\dots+w(g_r)}$ in $F(x)^n$ is also $\dfrac{n^r}{r!}$, thus completing the proof. To see this, first note that since we are assuming the $a(g)$s do not commute, the terms in $F(x)$ which could contribute to this coefficient in $F(x)^n$ are of the form 

$$\dfrac{1}{j!}a(g_i)a(g_{i+1})\dots a(g_{i+j-1})x^{w(g_i)+w(g_{i+1})+\dots +w(g_{i+j-1})}$$
\noindent
(in other words, consecutive blocks of $a(g_i)$s). But now, it follows directly from expansion that the coefficient of 

$$a(g_1)a(g_2)\dots a(g_r)x^{w(g_1)+w(g_2)+\dots+w(g_r)}$$
\noindent
in $F(x)^n$ is equal to

$$\sum_{m_1+\dots+m_s = r}\dfrac{1}{m_1!m_2!\dots m_s!}\binom{n}{s}$$
\noindent
which by Lemma \ref{binlemma} is simply equal to $\dfrac{n^r}{r!}$, as desired. (The ordered partitions arise from the different ways to divide the product $a(g_1)a(g_2)\dots a(g_r)$ into consecutive ``blocks''). 

\end{proof}

\begin{remark}
Interestingly enough, this proof works even if the $a(g)$ variables do not commute, so it holds even when $a$ is a function from $LCG(S)$ to $GL_n(\mathbb{C})$ (or any group). The author has not found any useful applications of this fact, however.
\end{remark}

We can now directly apply this theorem about intersection schemas to the case of non-overlapping placements. 

\begin{theorem}\label{mainresult}
Let $S$ be a (possibly infinite) set of connected $d$-dimensional figures, and let $p_{k}(n^d)$ be the number of ways to place some collection of these figures (possibly using the same figure in $S$ repeatedly) that have a total of $k$ edges on $T^{d}_n$. Then the $p_{k}(n^d)$ are eventually polynomials in $n^d$, and these polynomials form a sequence of binomial type.
\end{theorem}
\begin{proof}
Define the following intersection schema. Our weighted set $S$ is just the set $S$ of $d$-dimensional figures, where the weight of a figure is simply its number of edges. The function $a$ is defined as in equation (\ref{genA}); note that graphs in $LCG(S)$ are just connected overlap graphs for some set of figures. Then Theorem \ref{mainform} shows that for each $k$, $p_{k}(n^d)$ is eventually equal to $q_{k}(n^d)$. Our main theorem about intersection schemas (Theorem \ref{mainthm}) then shows that the polynomials $q_{k}(n^d)$ form a sequence of binomial type, as desired.
\end{proof}

\section{Generalizations}\label{generalizations}

Up until now, this paper has been concerned only with placing $d$-dimensional figures on $d$-dimensional toroidal grid graphs. However, the machinery of intersection schemas and inclusion-exclusion on overlap graphs allow us to prove a much wider range of results. In fact, it seems that any problem involving placing finite non-overlapping collections of subgraphs on larger and larger periodic graphs gives rise to an eventual polynomial sequence; in addition, if the underlying periodic graphs are (in some sense) ``toroidal'', then this polynomial sequence is of binomial type. In this section, we will consider some generalizations of Theorems \ref{polymain} and \ref{mainresult} that capture this idea.

\subsection{Other weights}

In the proof of Theorem \ref{mainresult}, we assigned the weight of a figure to be its number of edges. However, since Theorem \ref{mainthm} works for any valid weight function, we can essentially assign whatever weights we want to figures (as long as not too many figures have small weight). We can formalize this in the following statement.

\begin{theorem}\label{mainresultwts}
Let $S$ be a (possibly infinite) set of connected $d$-dimensional figures, and let $w$ be a function from $S$ to $\mathbb{Z}^{+}$ such that for any $x$, there are only finitely many figures $f \in S$ such that $w(f) = x$. Let $p_{k}(n^d)$ be the number of ways to place some collection of these figures (possibly using the same figure in $S$ repeatedly) that have a total of $k$ edges on $T^{d}_n$. Then the $p_{k}(n^d)$ are eventually polynomials, and these polynomials form a sequence of binomial type.
\end{theorem}

For example, all the following polynomial sequences are polynomial sequences of binomial-type:

\begin{itemize}
\item
Let $p_{k}(n^2)$ be the number of ways to place $a$ L-shaped triominos and $b$ T-shaped pentominos on an $n \times n$ toroidal grid such that $7a+2b=k$. Then $p_{k}(n^2)$ is eventually a polynomial sequence of binomial type.
\item
Let $p_{k}(n^d)$ be the number of ways to place some number of $d$-dimensional figures on a $d$-dimensional toroidal grid graph such that the sum of the squares of the number of edges over all figures equals $k$. Then $p_{k}(n^d)$ is eventually a polynomial sequence of binomial type.
\item
Let $p_{k}(n^d)$ be the number of ways to place some number of $d$-dimensional figures on a $d$-dimensional toroidal grid graph such that the total number of edges in the figures equals $k$, and then to color each figure that has at least $3$ edges with one of $50$ colors. Then $p_{k}(n^d)$ is eventually a polynomial sequence of binomial type.
\end{itemize}

\subsection{Continuous variant}

We can easily adapt intersection schemas to handle a continuous variant of our problem. To do this, we replace the $d$-dimensional toroidal grid graph $T^{d}_n$ with a continuous $d$-dimensional torus of side length $n$, and the concept of a $d$-dimensional figure with a bounded measurable set in $d$-dimensional Euclidean space. Then instead of counting the number of ways to place some number of objects so that they do not overlap, we instead consider the total measure of non-overlapping placements in state space. 

For the case where our collection of figures contains only one object, we have the following nice probabilistic result.

\begin{theorem}
Let $\mathcal{S}$ be a bounded measurable set in $d$-dimensional Euclidean space. Let $p_k(n^d)$ be the probability that no two copies intersect when we place $k$ copies of $\mathcal{S}$ independently and uniformly at random inside a $d$-dimensional torus of side-length $n$. Then $n^{dk}p_k(n^d)$ is eventually a polynomial for each $k$, and these polynomials form a sequence of binomial-type.
\end{theorem}

\subsection{Non-toroidal grids}

We can also ask what happens if, instead of placing our figures on the toroidal grid graph $T^{d}_n$, we place them on the regular grid graph $L^{d}_n$. It turns out that in this case we lose the binomial-type property. However, the number of possible placements is still a polynomial (in $n$ instead of $n^d$, however), and we therefore have the following analogue to Theorem \ref{polymain}. 

\begin{theorem}\label{polymainnt}
Let $S$ be a finite multiset of $d$-dimensional figures. Let $f_{S}(n)$ be the number of ways to place all of the figures in $S$ on $L^{d}_n$ such that none overlap. Then there exists a positive integer $n_0$ and an integer polynomial $p(x)$ such that $f_{S}(n) = p(n)$ for all $n \geq n_0$. 
\end{theorem}
\begin{proof}
We follow the proof of Theorem \ref{polymain}, with the slight change that for each configuration, instead of there being $n^d$ valid translations, there are only $(n-g_1)(n-g_2)\dots(n-g_d)$ valid translations, where $g_i$ is the girth of the configuration in dimension $i$ (that is, $g_i(c) = \max_{x, y \in c} |x_i - y_i|$). 
\end{proof}

We can extend this even farther. By the same reasoning as in the proof of Theorem \ref{polymain}, the above result holds for grid ``rectangles'' with unequal dimensions (like $n \times 2n$ rectangles, or $4n \times 5n \times 6n$ boxes). In fact, we have the following general result:

\begin{theorem}\label{polymainntgen}
Let $G$ be a graph formed by taking a finite subset of the unit $d$-dimensional cells comprising $T^{d}_{\infty}$. Let $G^n$ be the graph obtained by replacing each unit cell by a cell of length $n$ divided regularly into $n^d$ unit cells. Let $S$ be a finite multiset of $d$-dimensional figures, and let $f_{S}(n)$ be the number of ways to place all of the figures in $S$ on $G^n$ such that none overlap. Then there exists a positive integer $n_0$ and an integer polynomial $p(x)$ such that $f_{S}(n) = p(n)$ for all $n \geq n_0$.
\end{theorem}
\begin{proof}
Again, by following the same reasoning in \ref{polymain}, it suffices to show that the number of ways to place any one figure on $G^n$ is a polynomial in $n$.

To show this, call our figure $f$, and divide $G^n$ into copies of $L^{d}_n$ (in the same way that we can divide $G^1$ into unit $d$-dimensional cells). The number of ways to place $f$ in $L^{d}_n$ is a polynomial in $n$ (namely, the same polynomial used in the proof of Theorem \ref{polymainnt} above), so the total number of ways to place $f$ so that it stays entirely within one of these copies of $L^{d}_n$ is also a polynomial in $n$. Now, by similar reasoning, the number of ways this figure can intersect exactly $c$ of these $n$-dimensional cells is a polynomial in $n$ (since for each specific choice of $c$ cells, the number of ways this figure can intersect exactly those cells will be a polynomial in $n$). By summing all of these polynomials (and there are a finite number of these, since $G$ contains a finite number of unit cells), we find that the total number of ways to place $f$ in $G^n$ is a polynomial in $n$, as desired.
\end{proof}

\subsection{Other lattices}

Finally, the only discrete lattice we have considered is the square lattice. However, analogues of all of the above theorems exist for other lattices, such as triangular lattices and hexagonal lattices (and by exactly the same logic). 

\section{Application to Chromatic Polynomials}\label{chrom}

The following open problem appears as Exercise 4.82 in \textit{Enumerative Combinatorics, vol. 1}.

\begin{theorem}\label{chrommain}
Let $\chi_n(x)$ be the chromatic polynomial of the $n\times n$ toroidal grid graph, and let $q_k(n^2)$ be the coefficient of $x^{n^2-k}$ in $\chi_n(x)$. Then $(-1)^kq_k(n^2)$ is eventually a polynomial in $n^2$, and this sequence of polynomials is of binomial-type. 
\end{theorem}

In this section, we will provide a proof of this theorem, thus resolving this open problem. In addition, we will prove that the above claim holds not just for the $n \times n$ toroidal grid graph but for $T^{d}_n$, for any number of dimensions $d$. 

To do this, we will reduce the problem of computing the coefficient of $x^{n^2-k}$ in $\chi_n(x)$ to a placement problem, and then apply Theorem \ref{mainresult}. Our main tool for doing this will be Whitney's broken-circuit theorem, stated below.

\begin{definition}
In a graph $G = (V,E)$ with a total ordering on the edges, a \textit{broken circuit} is a subset of $E$ formed by taking a cycle in $G$ and removing the largest edge (with respect to the ordering).
\end{definition}

\begin{theorem}
\textbf{(Whitney's broken-circuit theorem)} \label{whitthm}
Let $G$ be a finite graph with a strict ordering on the edge set $E$. Then, for $n$ between $0$ and $|V|$ inclusive, the coefficient of $\lambda^{|V| - k}$ in $\chi_{G}(\lambda)$ is equal to $(-1)^{k}$ times the number of $k$-element subsets of $E$ which do not contain any broken-circuit of $G$ as a subset.
\end{theorem}
\begin{proof}
See \cite{Wh32}.
\end{proof}

It would be ideal if we could choose as our set $S$ of figures the set of connected $d$-dimensional figures which do not contain any broken-circuits. Unfortunately, the definition of broken-circuit depends on the ordering of the edges in the graph. Fortunately, we can choose an ordering of edges on $T^{d}_{n}$ that largely remedies this problem.

\begin{definition}
In the graph $T^{d}_{n}$, we call an ordering of the edge set $E$ \textit{natural} if it satisfies the following properties: 
\begin{enumerate}
\item
Each segment parallel to $u_i$ (where $u_i$ is the unit vector in dimension $i$) for $i\geq 2$ occurs before all edges parallel to $u_1$ (call these edges horizontal).
\item
If a horizontal edge $e$ connects points $(x_1, x_2, \dots, x_n)$ and points $(x_1+1, x_2, \dots, x_n)$, let the \textit{projection} of edge $e$, $p(e)$, be the $(n-1)$-tuple $(x_2, x_3, \dots, x_n)$. To compare two horizontal edges $e_1$ and $e_2$, let the larger edge be the edge with the lexicographically later projection vector.
\end{enumerate}

A natural ordering of the edge set of the graph $T^{d}_{\infty}$ is defined in the same way. 
\end{definition}

\begin{definition}
We say a figure is \textit{locally good} if its embedding in $T^{d}_{\infty}$ contains no broken-circuit under the natural edge ordering (note that if a translate of some subset of $T^{d}_{\infty}$ contains a broken-circuit iff the subset contains a broken-circuit, by construction of the natural edge ordering). We say that the placement of a figure in $T^{d}_{n}$ is \textit{globally good} if the corresponding subset of $T^{d}_n$ contains no broken-circuit under the natural edge ordering. If a figure is not locally/globally good, then it is locally/globally \textit{bad}.
\end{definition}

Now, we can let $S$ be the set of all locally good figures. However, note that it is possible to place a figure that is locally good on $T^{d}_{n}$ such that it is globally bad (for example, for $d = 2$, we can achieve this in certain cases by placing it so that it intersects the vertical line $x_2 = n$). Similarly, it is possible to place a figure which is locally bad on $T^{d}_{n}$ so that it is globally good. The following theorem will allow us to ignore such cases.

\begin{theorem} \label{convthm}
The total number of ways to place a globally bad cycle-free figure with $k$ edges on $T^{d}_{n}$ (over all possible figures with $k$ edges) is equal to the number of ways to place a locally bad cycle-free figure with $k$ edges on $T^{d}_{n}$.
\end{theorem}
\begin{proof}
We will exhibit a bijection between these two sets. Assume we have a figure $f$ (with $|E| = k$) which is cycle-free but globally bad. Since it is globally bad, it must contain some number of broken-circuits (under the natural edge ordering for $T^{d}_{n}$). Let the number of broken-circuits be $b$, and let $e_i$ be the edge needed to make the $i$th broken circuit a cycle. Note first that we cannot have $e_i = e_j$ for $i \neq j$, because if this were the case, then there would be two distinct paths between the endpoints of $e_i$ in $f$, which would imply that there is a cycle in $f$. Thus the $e_i$ comprise $b$ different edges. 

Let $\bar{f}$ be the graph formed by adding all of these edges to $f$ (so $\bar{f}$ now has $k+b$ edges). Now, consider $\bar{f}$ as a subgraph of $T^{d}_{\infty}$ with its natural edge ordering. Let $f'$ be the minimum spanning tree of $\bar{f}$, where we let the weight of the $r$th largest edge of $\bar{f}$ be $r$. We now claim that $f'$ is locally bad (it is cycle-free since it is a tree). To see this, note first that since $f'$ and $f$ are both spanning trees of $\bar{f}$, they both must have the same number $k$ of edges. Next, let $e'_1, e'_2, \dots, e'_b$ be the $b$ edges belonging to $\bar{f}$ but not to $f'$. Note that (by the properties of minimum spanning trees) if we add in $e'_i$ for any $i$, we will construct a unique simple cycle; moreover (again by the properties of minimum spanning trees), $e'_i$ will have the heaviest weight in this cycle. This implies that this set of edges in $f'$ (minus $e'_i$) forms a broken circuit under the local natural edge ordering, so $f'$ is locally bad (and in fact, it contains $b$ broken-circuits under this edge ordering). 

This procedure is a map which sends placements of globally bad cycle-free figures $f$ with $k$ edges to placements of locally bad cycle-free figures $f'$ with $k$ edges. Now, note that we can invert this map via the following procedure, thus showing that this map is a bijection. As before, we take the $b$ broken-circuits and the $b$ edges $e'_i$ required to make the $i$th broken circuit a cycle. We then add these $b$ edges to $f'$ to construct $\bar{f}$, and once we do this we let $f$ be the minimum spanning tree of $\bar{f}$ with respect to the natural edge ordering of $T^{d}_{n}$. To see that this restores the original $f$, note first that the figure $\bar{f}$ constructed in going from $f$ to $f'$ contains exactly the same edges as the figure $\bar{f}$ constructed in going back from $f'$ to $f$. Next, note that none of the edges $e_i$ can belong to the minimum spanning tree of $\bar{f}$ with respect to the natural edge ordering of $T^{d}_n$; this is since each such edge $e_i$ is the largest edge in a cycle, and such edges never occur in minimum spanning trees. But since $\bar{f}$ has $k+b$ edges, and there are $b$ edges $e_i$, this must mean that this minimum spanning tree is exactly $f$, as desired. 
\end{proof}

\begin{corollary} \label{lgcor}
The total number of ways to place a locally good cycle-free figure with $k$ edges on $T^{d}_{n}$  (over all possible figures with $k$ edges) is equal to the number of ways to place a globally good cycle-free figure with $k$ edges on $T^{d}_{n}$
\end{corollary}
\begin{proof}
Consider the following four sets of possible placements of figures with $k$ edges: $S_{GG}$, the set of locally good and globally good placements, $S_{GB}$, the set of locally good but globally bad placements, $S_{BG}$, the set of locally bad but globally good placements, and $S_{BB}$, the set of locally bad and globally bad placements. We wish to show that $|S_{GG}| + |S_{GB}| = |S_{GG}| + |S_{BG}|$, or equivalently, that $|S_{GB}| = |S_{BG}|$.

To do this, it suffices to show that $|S_{GB}| + |S_{BB}| = |S_{BG}| + |S_{BB}|$. Let $C$ be the set of placements of a figure with $k$ edges that has a cycle; note that any such placement must be both locally bad and globally bad, since any graph with a cycle contains a broken-circuit under any edge-ordering. We thus have that $C \subset S_{BB}$. Because of this, Theorem \ref{convthm} implies that $|S_{GB}| + |S_{BB}| - |C| = |S_{BG}| + |S_{BB}| - |C|$, and therefore that $|S_{GB}| + |S_{BB}| = |S_{BG}| + |S_{BB}|$, as desired.
\end{proof}

We can now prove the following generalization of Theorem \ref{chrommain}.

\begin{theorem}
Fix $d$, and let $\chi_n(x)$ be the chromatic polynomial of $T^{d}_n$. Let $q_k(n^2)$ be the coefficient of $x^{n^2-k}$ in $\chi_n(x)$. Then $(-1)^kq_k(n^d)$ is eventually a polynomial in $n^d$, and this sequence of polynomials is of binomial-type. 
\end{theorem}
\begin{proof}
By Whitney's broken-circuit theorem, $q_k(n^d)$ is equal to $(-1)^{k}$ times the number of $k$-element subsets of $T^{d}_n$ which contain no broken-circuit. By choosing a natural edge-ordering for $T^{d}_n$ and using the notation above, $(-1)^{k}q_k(n^d)$ is just the number of ways to place a globally good figure with $k$ edges on $T^{d}_n$. By Corollary \ref{lgcor}, this is equal to the number of ways to place a locally good figure with $k$ edges on $T^{d}_n$. By choosing $S$ to be the set of locally good connected figures, it follows from Theorem \ref{mainresult} that this number is indeed a polynomial in $n^d$ and that these polynomials form a sequence of binomial-type, as desired.
\end{proof}

\section{Acknowledgements}

This research was performed as part of MIT's Undergraduate Research Opportunities Program (UROP) in the summer of 2011. The author would like to thank Prof. Richard Stanley for introducing him to this problem, mentoring him over the course of this project, and helping edit this paper.

\end{document}